\documentclass[11pt, final]{article}
\usepackage{a4}
\usepackage{amsmath}%
\usepackage{amstext}%
\usepackage{amssymb}%
\usepackage{showkeys}%
\usepackage{epsfig}%
\usepackage{cite}
\usepackage{tikz}
\usepackage{subfig}
\usepackage{caption}
\usepackage{multirow}
\usepackage{booktabs}
\usepackage{floatrow}

\usepackage{listings}
\newtheorem{theorem}{Theorem}

\newtheorem{lemma}[theorem]{Lemma}

\newtheorem{pb}[theorem]{Problem}
\newenvironment{problem}{\pb\rm}{\endpb}

\newtheorem{rem}[theorem]{Remark}
\newenvironment{remark}{\rem\rm}{\endrem}

\newcounter{unnumber}

\newtheorem{prf}[unnumber]{Proof.}
\newenvironment{proof}{\prf\rm}{\hfill{$\blacksquare$}\endprf}
\newcommand{\R}{\mathbb{R}}%
\newcommand{\N}{\mathbb{N}}%
\newcommand{\ol}{\overline}%

\newcommand{\fx}{\ensuremath{\boldsymbol{x}}}

\newcommand{\fp}{\ensuremath{\boldsymbol{p}}}

\newcommand{\fM}{\ensuremath{\boldsymbol{M}}}
\newcommand{\fQ}{\ensuremath{\boldsymbol{Q}}}

\newcommand{\fK}{\ensuremath{\boldsymbol{\mathcal{K}}}}

\DeclareMathOperator*\inte{int}%
\DeclareMathOperator*\sqri{sqri}%
\DeclareMathOperator*\ri{ri}%
\DeclareMathOperator*\dom{dom}%
\DeclareMathOperator*\B{\overline{\R}}%
\DeclareMathOperator*\gr{Gr}%
\DeclareMathOperator*\ran{ran}%
\DeclareMathOperator*\id{Id}%
\DeclareMathOperator*\prox{prox}%

\DeclareMathOperator*\zer{zer}

\title{An inertial forward-backward-forward primal-dual splitting algorithm for solving monotone inclusion problems}

\author{Radu Ioan Bo\c{t} \thanks{University of Vienna, Faculty of Mathematics, Oskar-Morgenstern-Platz 1, A-1090 Vienna, Austria,
email: radu.bot@univie.ac.at. Research partially supported by DFG (German Research Foundation), project BO 2516/4-1.} \and
Ern\"{o} Robert Csetnek \thanks {University of Vienna, Faculty of Mathematics, Oskar-Morgenstern-Platz 1, A-1090 Vienna, Austria,
email: ernoe.robert.csetnek@univie.ac.at. Research supported by DFG (German Research Foundation), project BO 2516/4-1.}}

\begin{document}
\maketitle

\noindent \textbf{Abstract.} We introduce and investigate the convergence properties of an inertial forward-backward-forward splitting algorithm for approaching the set of zeros of the
sum of a maximally monotone operator and a single-valued monotone and Lipschitzian operator. By making use of the product space approach, we expand it to the solving of inclusion problems involving mixtures of linearly
composed and parallel-sum type monotone operators. We obtain in this way an inertial forward-backward-forward primal-dual  splitting algorithm having as main characteristic the fact that in the iterative scheme all operators are
accessed separately either via forward or via backward evaluations.  We present also the variational case when one is interested in the solving of a primal-dual pair of convex optimization problems with
intricate objective functions.
\vspace{1ex}

 \vspace{1ex}


\section{Introduction and preliminaries}\label{sec-intr}

Due to its wide applicability in different branches of the applied mathematics, especially in connection with real-life problems, the problem of solving inclusion problems
involving mixtures of monotone operators in Hilbert spaces continues to attract the interest of many researchers (see \cite{bauschke-book, b-c-h1, b-c-h1, b-h, b-h2, br-combettes, combettes, combettes-pesquet, vu}).

In this paper we will focus on the class of so-called \textit{inertial proximal methods}, the origins of which go back to \cite{alvarez2000, alvarez-attouch2001}.
The idea behind the iterative scheme relies on the use of an implicit discretization of a differential system of second-order in time and it was employed for the first time in  the context of finding the zeros of a
maximally monotone operator in \cite{alvarez-attouch2001}. One  of the main features of the inertial proximal algorithm is that the next iterate is defined by making use of the last two iterates. It also turns out that
the method is a generalization of the classical proximal-point one (see \cite{rock-prox}). Since its introduction, one can notice an increasing interest in the class of inertial type algorithms, see \cite{alvarez2000, alvarez-attouch2001, alvarez-attouch2001,
cabot-frankel2011, mainge2008, mainge-moudafi2008, moudafi-oliny2003}. Especially noticeable is that these ideas where also used in the context of determining the zeros of the sum of a maximally monotone operator and a
(single-valued) cocoercive operator, giving rise to the so-called inertial forward-backward algorithm \cite{moudafi-oliny2003}. This is an extension of the classical \textit{forward-backward algorithm}
(see \cite{bauschke-book, combettes}) and assumes the evaluation of the set-valued operator via its resolvent, called backward step, while the single-valued operator is evaluated via a forward step.

The first major aim of this manuscript to introduce and investigate an inertial forward-backward-forward splitting algorithm for finding
the zeros of the sum of a maximally monotone operator and a monotone and Lipschitzian operator. The proposed scheme represents an extension of Tseng's \textit{forward-backward-forward-type algorithm},
(see \cite{bauschke-book, br-combettes, tseng, Tse91}), however, for the study of its convergence properties we will use some generalizations of the
Fej\'{e}r monotonicity techniques provided in \cite{alvarez-attouch2001}. An essential argument in the favor of forward-backward-forward splitting algorithms is given by the fact that they can be used when solving
a larger class of monotone inclusion problems, since it is known that there exist monotone and Lipschitzian  operators which are not cocoercive, in which case the forward-backward algorithms cannot be applied
(see \cite{br-combettes, combettes-pesquet, b-c-h1}). This is for instance the case when considering primal-dual splitting methods, as one can notice by consulting \cite{br-combettes, combettes-pesquet, b-c-h1}.

\textit{Primal-dual splitting algorithms} are modern techniques designed to solve inclusion problems where some complex structures
of monotone operators are involved, such as mixtures of linearly
composed and parallel-sum type monotone operators. The key feature of these algorithms is that
they are fully decomposable, in the sense that each of the operators  are evaluated in the algorithm separately, either via forward or via backward steps.
It is also noticeable that the primal-dual algorithms solve concomitantly a (primal) monotone inclusion problem and its dual monotone inclusion problem in the sense of
Attouch-Th\'{e}ra \cite{at-th}. We invite the reader to consult  \cite{b-c-h1, b-c-h2, b-h, b-h2, br-combettes, combettes-pesquet, vu, ch-pck, condat2013}  for further considerations concerning this class of algorithms.
The second major aim of this paper will be to formulate an inertial primal-dual splitting algorithm relying on the inertial forward-backward-forward one.

The structure of the paper is the following. The remainder of this section is dedicated
to some elements of the theory of maximal monotone operators and to the recall of some convergence results. In the next section we formulate the inertial forward-backward-forward splitting algorithm for finding
the zeros of the sum of a maximally monotone operator and a monotone and Lipschitzian operator and investigate its convergence. In Section \ref{sec3} we use the product space approach in order to obtain the inertial primal-dual splitting algorithm designed for solving monotone inclusion problems involving 
mixtures of linearly composed and parallel-sum type monotone operators. Finally, we show how the proposed iterative schemes can be used in order to solve primal-dual pairs of convex optimization problems.

For the notions and results presented as follows we refer the reader to \cite{bo-van, b-hab, bauschke-book, EkTem, simons, Zal-carte}. Let $\N= \{0,1,2,...\}$ be the set of nonnegative integers.
Let ${\cal H}$ be a real Hilbert space with \textit{inner product} $\langle\cdot,\cdot\rangle$ and associated \textit{norm} $\|\cdot\|=\sqrt{\langle \cdot,\cdot\rangle}$.
The symbols $\rightharpoonup$ and $\rightarrow$ denote weak and strong convergence, respectively.
When ${\cal G}$ is another Hilbert space and $K:{\cal H} \rightarrow {\cal G}$ a linear continuous operator,
then the \textit{norm} of $K$ is defined as $\|K\| = \sup\{\|Kx\|: x \in {\cal H}, \|x\| \leq 1\}$,
while $K^* : {\cal G} \rightarrow {\cal H}$, defined by $\langle K^*y,x\rangle = \langle y,Kx \rangle$ for all $(x,y) \in {\cal H} \times {\cal G}$, denotes the \textit{adjoint operator} of $K$.

For an arbitrary set-valued operator $A:{\cal H}\rightrightarrows {\cal H}$ we denote by $\gr A=\{(x,u)\in {\cal H}\times {\cal H}:u\in Ax\}$ its \emph{graph}, by $\dom A=\{x \in {\cal H} : Ax \neq \emptyset\}$
its \emph{domain}, by $\ran A=\cup_{x\in{\cal{H}}} Ax$ its {\it range} and by $A^{-1}:{\cal H}\rightrightarrows {\cal H}$ its
\emph{inverse operator}, defined by $(u,x)\in\gr A^{-1}$ if and only if $(x,u)\in\gr A$.
We use also the notation $\zer A=\{x\in{\cal{H}}:0\in Ax\}$ for the \emph{set of zeros} of $A$. We say that $A$ is \emph{monotone}
if $\langle x-y,u-v\rangle\geq 0$ for all $(x,u),(y,v)\in\gr A$. A monotone operator $A$ is said to be \emph{maximally monotone}, if there exists no proper monotone extension of the graph of $A$ on ${\cal H}\times {\cal H}$.
The \emph{resolvent} of $A$, $J_A:{\cal H} \rightrightarrows {\cal H}$, is defined by $J_A=(\id_{{\cal H}}+A)^{-1}$, where $\id_{{\cal H}} :{\cal H} \rightarrow {\cal H}, \id_{\cal H}(x) = x$ for all $x \in {\cal H}$, is the \textit{identity operator} on ${\cal H}$. Moreover, if $A$ is maximally monotone, then $J_A:{\cal H} \rightarrow {\cal H}$ is single-valued and maximally monotone
(see \cite[Proposition 23.7 and Corollary 23.10]{bauschke-book}). For an arbitrary $\gamma>0$ we have (see \cite[Proposition 23.2]{bauschke-book})
$$p\in J_{\gamma A}x \ \mbox{if and only if} \ (p,\gamma^{-1}(x-p))\in\gr A$$
and (see \cite[Proposition 23.18]{bauschke-book})
\begin{equation}\label{j-inv-op}
J_{\gamma A}+\gamma J_{\gamma^{-1}A^{-1}}\circ \gamma^{-1}\id\nolimits_{{\cal H}}=\id\nolimits_{{\cal H}}.
\end{equation}

Further, let us mention some classes of operators that are used in the paper. We say that $A$ is \textit{demiregular} at $x\in\dom A$ if, for every
sequence $(x_n,u_n)_{n\in\N}\in\gr A$ and every $u \in Ax$ such that $x_n\rightharpoonup x$ and $u_n\rightarrow u$, we have
$x_n\rightarrow x$. We refer the reader to \cite[Proposition 2.4]{Att-Ar-Com10} and \cite[Lemma 2.4]{br-combettes} for
conditions ensuring this property. The operator $A$ is said to be \textit{uniformly monotone} at $x\in\dom A$ if there exists an increasing function
$\phi_A : [0,+\infty) \rightarrow [0,+\infty]$ that vanishes only at $0$, and
$\langle x-y,u-v \rangle \geq \phi_A \left( \| x-y \|\right)$ for every $u \in Ax$ and $(y,v) \in \gr A$. If this inequality
holds for all $(x,u),(y,v) \in \gr A$, we say that $A$ is uniformly monotone. If $A$ is uniformly monotone at $x \in \dom A$, then it is demiregular at $x$.

Prominent representatives of the class of uniformly monotone operators are the strongly monotone operators.
Let $\gamma>0$ be arbitrary. We say that  $A$ is \textit{$\gamma$-strongly monotone}, if $\langle x-y,u-v\rangle\geq \gamma\|x-y\|^2$ for all $(x,u),(y,v)\in\gr A$.
Further, a single-valued operator $A:{\cal H}\rightarrow {\cal H}$ is said to be \textit{$\gamma$-cocoercive} if $\langle x-y,Ax-Ay\rangle\geq \gamma\|Ax-Ay\|^2$ for all $(x,y)\in {\cal H}\times {\cal H}$.
Moreover, $A$ is \textit{$\gamma$-Lipschitzian} if $\|Ax-Ay\|\leq \gamma\|x-y\|$ for all $(x,y)\in {\cal H}\times {\cal H}$. A single-valued linear operator $A:{\cal H} \rightarrow {\cal H}$ is said to be \textit{skew}, if $\langle x,Ax\rangle =0$ for all $x \in {\cal H}$. Finally, the \textit{parallel sum} of two operators $A,B:{\cal H}\rightrightarrows {\cal H}$ is defined by $A\Box B: {\cal H}\rightrightarrows {\cal H}, A\Box B=(A^{-1}+B^{-1})^{-1}$.

We close this section by presenting three convergence results which will be crucial for the proof of the main results
in the next section.

\begin{lemma}\label{ext-fejer1} (see \cite{alvarez-attouch2001, alvarez2000, alvarez2004}) Let $(\varphi_n)_{n\in\N}, (\delta_n)_{n\in\N}$ and $(\alpha_n)_{n\in \N}$ be sequences in
$[0,+\infty)$ such that $\varphi_{n+1}\leq\varphi_n+\alpha_n(\varphi_n-\varphi_{n-1})+\delta_n$
for all $n \geq 1$, $\sum_{n\in \N}\delta_n< + \infty$ and there exists a real number $\alpha$ with
$0\leq\alpha_n\leq\alpha<1$ for all $n\in\N$. Then the following hold: \begin{itemize}\item[(i)] $\sum_{n \geq 1}[\varphi_n-\varphi_{n-1}]_+< + \infty$, where
$[t]_+=\max\{t,0\}$; \item[(ii)] there exists $\varphi^*\in[0,+\infty)$ such that $\lim_{n\rightarrow+\infty}\varphi_n=\varphi^*$.\end{itemize}
\end{lemma}

An easy consequence of Lemma \ref{ext-fejer1} is the following result.

\begin{lemma}\label{ext-fejer2} Let $(\varphi_n)_{n\in\N}, (\delta_n)_{n\in\N}, (\alpha_n)_{n\in \N}$ and $(\beta_n)_{n\in \N}$ be sequences in
$[0,+\infty)$ such that $\varphi_{n+1}\leq-\beta_n+\varphi_n+\alpha_n(\varphi_n-\varphi_{n-1})+\delta_n$
for all $n \geq 1$, $\sum_{n\in \N}\delta_n<+\infty$ and there exists a real number $\alpha$ with
$0\leq\alpha_n\leq\alpha<1$ for all $n\in\N$. Then the following hold: \begin{itemize}\item[(i)] $\sum_{n \geq 1}[\varphi_n-\varphi_{n-1}]_+<+\infty$, where
$[t]_+=\max\{t,0\}$; \item[(ii)] there exists $\varphi^*\in[0,+\infty)$ such that $\lim_{n\rightarrow + \infty}\varphi_n=\varphi^*$; \item[(iii)] $\sum_{n\in\N}\beta_n<+\infty$.\end{itemize}
\end{lemma}

Finally, we recall a well known result on weak convergence in Hilbert spaces.

\begin{lemma}\label{opial} (Opial) Let $C$ be a nonempty set of ${\cal H}$ and $(x_n)_{n\in\N}$ be a sequence in ${\cal H}$ such that
the following two conditions hold: \begin{itemize}\item[(a)] for every $x\in C$, $\lim_{n\rightarrow + \infty}\|x_n-x\|$ exists;
\item[(b)] every sequential weak cluster point of $(x_n)_{n\in\N}$ is in $C$;\end{itemize}
Then $(x_n)_{n\in\N}$ converges weakly to a point in $C$.
 \end{lemma}

\section{An inertial forward-backward-forward splitting algorithm}\label{sec2}

This section is dedicated to the formulation of an inertial forward-backward-forward splitting algorithm which approaches the set of zeros
of the sum of two maximally monotone operators, one of them being single-valued and Lipschitzian, and to the investigation of its convergence properties.

\begin{theorem}\label{inertial-tseng} Let $A:{\cal H}\rightrightarrows {\cal H}$ be a maximally monotone operator and $B:{\cal H}\rightarrow {\cal H}$ a
monotone and $\beta$-Lipschitzian operator for some $\beta>0$. Suppose that $\zer(A+B)\neq\emptyset$ and consider the following iterative scheme:
$$(\forall n\geq 1)\hspace{0.2cm}\left\{
\begin{array}{ll}
p_n=J_{\lambda_n A}[x_n-\lambda_nBx_n+\alpha_{1,n}(x_n-x_{n-1})]\\
x_{n+1}=p_n+\lambda_n(Bx_n-Bp_n)+\alpha_{2,n}(x_n-x_{n-1}),
\end{array}\right.$$ where $x_0$ and $x_1$ are arbitrarily chosen in ${\cal H}$. Consider $\lambda, \sigma>0$ and
$\alpha_1, \alpha_2\geq 0$  such that
\begin{equation}\label{hypotheses}
12\alpha_2^2+9(\alpha_1+\alpha_2)+4\sigma<1 \ \mbox{and} \ \lambda\leq\lambda_n\leq \frac{1}{\beta}\sqrt{\frac{1-12\alpha_2^2-9(\alpha_1+\alpha_2)-4\sigma}{12\alpha_2^2+8(\alpha_1+\alpha_2)+4\sigma +2}} \ \forall n \geq 1
\end{equation}
and for $i=1,2$ the nondecreasing sequences $(\alpha_{i,n})_{n \geq 1}$ with $\alpha_{i,1}=0$ and $0\leq\alpha_{i,n}\leq\alpha_i$ for all $n \geq 1$.
Then there exists $\ol x\in\zer(A+B)$ such that the following statements are true:
\begin{itemize}\item[(a)] $\sum_{n\in\N}\|x_{n+1}-x_n\|^2<+\infty$
and $\sum_{n \geq 1}\|x_n-p_n\|^2<+\infty$; \item[(b)] $x_n\rightharpoonup\ol x$ and $p_n\rightharpoonup\ol x$ as $n \rightarrow +\infty$; \item[(c)] Suppose that
one of the following conditions is satisfied: \begin{itemize}\item[(i)]$A+B$ is demiregular at $\ol x$; \item[(ii)] $A$ or $B$ is
uniformly monotone at $\ol x$.\end{itemize} Then $x_n\rightarrow\ol x$ and $p_n\rightarrow\ol x$ as $n \rightarrow +\infty$.\end{itemize}

\end{theorem}

\begin{proof}
Let $z$ be a fixed element in $\zer(A+B)$, that is $-Bz\in Az$, and $n \geq 1$. From the definition of the resolvent we deduce
$$\frac{1}{\lambda_n}(x_n-p_n)-Bx_n+\frac{\alpha_{1,n}}{\lambda_n}(x_n-x_{n-1})\in Ap_n.$$ Further, taking into account the
relation between $p_n$ and $x_{n+1}$ in the algorithm, we obtain
\begin{equation}\label{def-res}\frac{1}{\lambda_n}(x_n-x_{n+1})-Bp_n+\frac{\alpha_{1,n}+\alpha_{2,n}}{\lambda_n}(x_n-x_{n-1})\in Ap_n.\end{equation}
The monotonicity of $A$ delivers the inequality
$$0\leq \left\langle \frac{1}{\lambda_n}(x_n-x_{n+1})-Bp_n+\frac{\alpha_{1,n}+\alpha_{2,n}}{\lambda_n}(x_n-x_{n-1})+Bz,p_n-z\right\rangle,$$
hence
\begin{equation}\label{ineq1}0\leq \frac{1}{\lambda_n}\langle x_n-x_{n+1},p_n-z\rangle+\langle Bz-Bp_n,p_n-z\rangle+
\frac{\alpha_{1,n}+\alpha_{2,n}}{\lambda_n}\langle x_n-x_{n-1},p_n-z\rangle.\end{equation}

Since $B$ is monotone, we have $\langle Bz-Bp_n,p_n-z\rangle\leq 0$. Moreover,
$$\langle x_n-x_{n+1},p_n-z\rangle=\langle x_n-x_{n+1},p_n-x_{n+1}\rangle+\langle x_n-x_{n+1},x_{n+1}-z\rangle=$$$$
\frac{\|x_n-x_{n+1}\|^2}{2}+\frac{\|p_n-x_{n+1}\|^2}{2}-\frac{\|x_n-p_n\|^2}{2}+\frac{\|x_n-z\|^2}{2}
-\frac{\|x_n-x_{n+1}\|^2}{2}-\frac{\|x_{n+1}-z\|^2}{2}.$$
In a similar way we obtain $$\langle x_n-x_{n-1},p_n-z\rangle=\langle x_n-x_{n-1},x_n-z\rangle+\langle x_n-x_{n-1},p_n-x_n\rangle=$$$$
\frac{\|x_n-x_{n-1}\|^2}{2}+\frac{\|x_n-z\|^2}{2}-\frac{\|x_{n-1}-z\|^2}{2}+\frac{\|p_n-x_{n-1}\|^2}{2}
-\frac{\|x_n-x_{n-1}\|^2}{2}-\frac{\|x_n-p_n\|^2}{2}.$$

Further we have, by using that $B$ is $\beta$-Lipschitzian,
$$\|x_{n+1}-p_n\|^2\leq 2\lambda_n^2\beta^2\|x_n-p_n\|^2+2\alpha_{2,n}^2\|x_n-x_{n-1}\|^2$$
and
$$\|p_n-x_{n-1}\|^2\leq 2\|x_n-p_n\|^2+2\|x_n-x_{n-1}\|^2.$$
The above estimates together with \eqref{ineq1} imply
\begin{align*}
0\leq & \left(\frac{1}{2\lambda_n}+\frac{\alpha_{1,n}+\alpha_{2,n}}{2\lambda_n}\right)\|x_n-z\|^2-\frac{1}{2\lambda_n}\|x_{n+1}-z\|^2-
\frac{\alpha_{1,n}+\alpha_{2,n}}{2\lambda_n}\|x_{n-1}-z\|^2+\\
& \left(\lambda_n\beta^2-\frac{1}{2\lambda_n}+\frac{\alpha_{1,n}+\alpha_{2,n}}{\lambda_n}-
\frac{\alpha_{1,n}+\alpha_{2,n}}{2\lambda_n}\right)\|x_n-p_n\|^2 +\\
& \left(\frac{\alpha_{2,n}^2}{\lambda_n}+\frac{\alpha_{1,n}+\alpha_{2,n}}{\lambda_n}\right)\|x_n-x_{n-1}\|^2,
\end{align*}
from which we further obtain, after multiplying with $2\lambda_n$,
\begin{eqnarray}\label{ineq2}
& &\|x_{n+1}-z\|^2-(1+\alpha_{1,n}+\alpha_{2,n})\|x_n-z\|^2+(\alpha_{1,n}+\alpha_{2,n})\|x_{n-1}-z\|^2 \leq \nonumber\\
& &-(1-\alpha_{1,n}-\alpha_{2,n}-2\lambda_n^2\beta^2)\|x_n-p_n\|^2+2(\alpha_{2,n}^2+\alpha_{1,n}+\alpha_{2,n})\|x_n-x_{n-1}\|^2.\end{eqnarray}

By using the bounds given for the sequences $(\lambda_n)_{n \geq 1}$, $(\alpha_{1,n})_{n \geq 1}$ and $(\alpha_{2,n})_{n \geq 1}$ one can easily show by taking into account \eqref{hypotheses} that
$$2\lambda_n^2\beta^2 < 1- \alpha_1 - \alpha_2 \leq 1-\alpha_{1,n}-\alpha_{2,n},$$
thus
$$1-\alpha_{1,n}-\alpha_{2,n} - 2\lambda_n^2\beta^2 > 0.$$
Taking into account that
\begin{align*}
\|x_{n+1}-x_n\|^2 = & \ \|p_n-x_n+\lambda_n(Bx_n-Bp_n)+\alpha_{2,n}(x_n-x_{n-1})\|^2\\
\leq & \ 2(1+\lambda_n\beta)^2\|x_n-p_n\|^2+2\alpha_{2,n}^2\|x_n-x_{n-1}\|^2,
\end{align*}
we obtain from \eqref{ineq2}
\begin{align}\label{ineq3}
& \|x_{n+1}-z\|^2-(1+\alpha_{1,n}+\alpha_{2,n})\|x_n-z\|^2+(\alpha_{1,n}+\alpha_{2,n})\|x_{n-1}-z\|^2 \leq \nonumber\\
& -\frac{1-\alpha_{1,n}-\alpha_{2,n}-2\lambda_n^2\beta^2}{2(1+\lambda_n\beta)^2}\|x_{n+1}-x_n\|^2+\gamma_n\|x_n-x_{n-1}\|^2,
\end{align}
where
$$\gamma_n:=2(\alpha_{2,n}^2+\alpha_{1,n}+\alpha_{2,n})+\frac{\alpha_{2,n}^2(1-\alpha_{1,n}-\alpha_{2,n}-2\lambda_n^2\beta^2)}{(1+\lambda_n\beta)^2} >0.$$

(a) For the proof of this statement we are going to use some techniques from \cite{alvarez-attouch2001}. We define the sequences
$\varphi_n:=\|x_n-z\|^2$ for all $n\in\N$ and $\mu_n:=\varphi_n-(\alpha_{1,n}+\alpha_{2,n})\varphi_{n-1}+\gamma_n\|x_n-x_{n-1}\|^2$ for all
$n\geq 1$. Using the monotonicity of $(\alpha_{i,n})_{n \geq 1}$, $i=1,2,$ and the fact that $\varphi_n\geq0$ for all $n\in\N$ we get
\begin{align*}
& \mu_{n+1}-\mu_n\leq \\
& \varphi_{n+1}-(1+\alpha_{1,n}+\alpha_{2,n})\varphi_n+(\alpha_{1,n}+\alpha_{2,n})\varphi_{n-1}+\gamma_{n+1}\|x_{n+1}-x_n\|^2-\gamma_n\|x_n-x_{n-1}\|^2,
\end{align*}
which gives by \eqref{ineq3}
\begin{equation}\label{ineq4}
\mu_{n+1}-\mu_n\leq-\left(\frac{1-\alpha_{1,n}-\alpha_{2,n}-2\lambda_n^2\beta^2}{2(1+\lambda_n\beta)^2}-\gamma_{n+1}\right)\|x_{n+1}-x_n\|^2 \ \forall n \geq 1.
\end{equation}

We claim that \begin{equation}\label{ineq5}\frac{1-\alpha_{1,n}-\alpha_{2,n}-2\lambda_n^2\beta^2}{2(1+\lambda_n\beta)^2}-\gamma_{n+1}\geq \sigma \ \forall n \geq 1.\end{equation}
Indeed, this follows by taking into account that for all $n \geq 1$
\begin{align*}
& \alpha_{1,n}+\alpha_{2,n}+2(\lambda_n\beta)^2+2(1+\lambda_n\beta)^2(\gamma_{n+1}+\sigma)\leq \\
& \alpha_1+\alpha_2+2(\lambda_n\beta)^2+2(1+\lambda_n\beta)^2(3\alpha_2^2+2(\alpha_1+\alpha_2)+\sigma) \leq \\
& \alpha_1+\alpha_2+2(\lambda_n\beta)^2+4(1+(\lambda_n\beta)^2)(3\alpha_2^2+2(\alpha_1+\alpha_2)+\sigma)\leq 1.
\end{align*}
In the above estimates we used the upper bounds for $(\alpha_{i,n})_{n \geq 1}$, $i=1,2$, that
$$\gamma_{n+1}\leq 2(\alpha_2^2+\alpha_1+\alpha_2)+\alpha_2^2 \ \forall n \in \N$$
and the assumptions in \eqref{hypotheses}.

We obtain from \eqref{ineq4} and \eqref{ineq5} that
\begin{equation}\label{ineq6}\mu_{n+1}-\mu_n\leq-\sigma\|x_{n+1}-x_n\|^2 \ \forall n\geq 1.\end{equation}

The sequence $(\mu_n)_{n \geq 1}$ is nonincreasing and the bounds for $(\alpha_{i,n})_{n \geq 1}$, $i=1,2$, deliver
\begin{equation}\label{ineq7}
-(\alpha_1+\alpha_2)\varphi_{n-1}\leq\varphi_n-(\alpha_1+\alpha_2)\varphi_{n-1}\leq \mu_n\leq\mu_1 \ \forall n \geq 1.
\end{equation}

We obtain
$$\varphi_n\leq (\alpha_1+\alpha_2)^n\varphi_0+\mu_1\sum_{k=0}^{n-1}(\alpha_1+\alpha_2)^k\leq (\alpha_1+\alpha_2)^n\varphi_0+\frac{\mu_1}{1-\alpha_1-\alpha_2} \ \forall n \geq 1,$$
where we notice that $\mu_1=\varphi_1\geq 0$ (due to the relation $\alpha_{1,1}=\alpha_{2,1}=0$). Combining \eqref{ineq6} and \eqref{ineq7} we get for all $n \geq 1$
$$\sigma\sum_{k=1}^{n}\|x_{k+1}-x_k\|^2\leq \mu_1-\mu_{n+1}\leq \mu_1+(\alpha_1+\alpha_2)\varphi_n\leq (\alpha_1+\alpha_2)^{n+1}\varphi_0+\frac{\mu_1}{1-\alpha_1-\alpha_2},$$ which
shows that $\sum_{n \in \N}\|x_{n+1}-x_n\|^2<+\infty$.

Combining this relation with \eqref{ineq2} and Lemma \ref{ext-fejer2} it yields
$$\sum_{n \geq 1}\left(1-\alpha_{1,n}-\alpha_{2,n}-2\lambda_n^2\beta^2\right)\|x_n-p_n\|^2<+\infty.$$
Moreover, from \eqref{ineq5} we have $1-\alpha_{1,n}-\alpha_{2,n}-2\lambda_n^2\beta^2\geq 2\sigma(1+\lambda\beta)^2$ for all $n \geq 1$ and obtain, consequently, $\sum_{n \geq 1} \|x_n-p_n\|^2< +\infty$.

(b) We are going to use Lemma \ref{opial}. We proved above that for an arbitrary $z\in\zer(A+B)$ the inequality \eqref{ineq2} is true. By
part (a) and Lemma \ref{ext-fejer2} it follows that $\lim_{n\rightarrow+\infty}\|x_n-z\|$ exists. On the other hand, let $x$ be a sequential weak cluster point of
$(x_n)_{n\in\N}$, that is, it has a subsequence $(x_{n_k})_{k \in \N}$ fulfilling $x_{n_k}\rightharpoonup x$ as $k\rightarrow+\infty$. Since $x_n-p_n\rightarrow 0$ as $n\rightarrow+\infty$, 
we get $p_{n_k}\rightharpoonup x$ as $k\rightarrow+\infty$. Since $A+B$ is maximally monotone
(see \cite[Corollary 20.25 and Corollary 24.4]{bauschke-book}), its graph is sequentially closed in
the weak-strong topology of ${\cal H}\times {\cal H}$ (see \cite[Proposition 20.33(ii)]{bauschke-book}).
As $(\lambda_n)_{n \geq 1}$ and $(\alpha_{i,n})_{n \geq 1}$, $i=1,2$, are bounded, we derive from \eqref{def-res} and part (a) that $0\in (A+B)x$, hence $x\in\zer(A+B)$. 
By Lemma \ref{opial} there exists $\ol x\in\zer(A+B)$ such that $x_n\rightharpoonup \ol x$ as $n\rightarrow +\infty$. In view of (a) we have
$p_n\rightharpoonup \ol x$ as $n \rightarrow +\infty$.

(c) Since (ii) implies that $A+B$ is uniformly monotone at $\ol x$, hence demiregular at $\ol x$, it is sufficient to prove the statement under condition (i).
Since $p_n\rightharpoonup\ol x$ and $\frac{1}{\lambda_n}(x_n-x_{n+1})+\frac{\alpha_{1,n}+\alpha_{2,n}}{\lambda_n}(x_n-x_{n-1})\rightarrow 0$ as $n \rightarrow +\infty$, the result follows easily from \eqref{def-res} and the definition of demiregular operators.
\end{proof}

\begin{remark}\label{xo-x1} Let us mention that the conclusion of the theorem holds also in case one assumes that the sequence
$(\alpha_{1,n}+\alpha_{2,n})_{n \geq 1}$ is nondecreasing. Moreover, the condition $\alpha_{1,1}=\alpha_{2,1}=0$ was imposed in order to ensure $\mu_1\geq 0$, which is needed in the proof.
An alternative is to require that $x_0=x_1$, in which case the assumption $\alpha_{1,1}=\alpha_{2,1}=0$ is not anymore necessary.
\end{remark}

\begin{remark}\label{tseng-classic-error-free} Assuming that $\alpha_2=0$, which enforces $\alpha_{2,n}=0$ for all $n \geq 1$, the conclusions of Theorem \ref{inertial-tseng} remains valid if one takes as upper bound for $(\lambda_n)_{n \geq 1}$ the expression $\frac{1}{\beta}\sqrt{\frac{1-5\alpha_1-2\sigma}{4\alpha_1+2\sigma+1}}$. This is due to the fact in this situation one can use in its proof the improved inequalities $\|x_{n+1}-p_n\|^2\leq \lambda_n^2\beta^2\|x_n-p_n\|^2$ and $\|x_{n+1}-x_n\|^2\leq(1+\lambda_n\beta)^2\|x_n-p_n\|^2$ for all $n \geq 1$. On the other hand, let us also notice that the algorithmic scheme obtained in this way and its convergence properties can be seen as generalizations of the corresponding statements given for the
error-free case of the classical forward-backward-forward algorithm proposed by Tseng in \cite{tseng} (see also \cite[Theorem 2.5]{br-combettes}).  Indeed, if we further set $\alpha_1=0$, 
having as consequence that $\alpha_{1,n}=0$ for all $n \geq 1$, we obtain nothing else than the iterative scheme from \cite{tseng, br-combettes}. Notice that for $\varepsilon\in(0,1/(\beta+1))$, 
one can chose $\lambda:=\varepsilon$ and $\sigma:=\frac{1-(1-\varepsilon)^2}{2(1+(1-\varepsilon)^2)}$. In this case the sequence $(\lambda_n)_{n \geq 1}$ must fulfill the inequalities $\varepsilon\leq\lambda_n\leq \frac{1}{\beta}\sqrt{\frac{1-2\sigma}{2\sigma+1}}=\frac{1-\varepsilon}{\beta}$ for all
$n \geq 1$, which is exactly the situation considered in \cite{br-combettes}.
\end{remark}

\begin{remark}\label{inertial-prox} In case $Bx=0$ for all $x\in\cal{H}$ the proposed iterative scheme becomes
$$x_{n+1}=J_{\lambda_n A}[x_n+\alpha_{1,n}(x_n-x_{n-1})]+\alpha_{2,n}(x_n-x_{n-1}) \ \forall n \geq 1,$$ 
and is to the best of our knowledge new and can be regarded as an extension of the classical proximal-point algorithm (see \cite{rock-prox}) in the context of solving the
monotone inclusion problem $0 \in Ax$. If, additionally, $\alpha_2=0$,  which enforces as already noticed $\alpha_{2,n}=0$ for all $n \geq 1$, we get the algorithm
$$x_{n+1}=J_{\lambda_n A}[x_n+\alpha_{1,n}(x_n-x_{n-1})],$$
the convergence of which has been investigated in \cite{alvarez-attouch2001}.
\end{remark}

\section{Solving monotone inclusion problems involving mixtures of linearly composed and parallel-sum type operators}\label{sec3}

In this section we employ the inertial forward-backward-forward splitting algorithm proposed above to the concomitantly solving of a primal monotone inclusion problem
involving mixtures of linearly composed and parallel-sum type operators and its Attouch-Th\'{e}ra-type dual problem. We consider the following setting.

\begin{problem}\label{pr1}
Let ${\cal H}$ be a real Hilbert space, $z\in {\cal H}$, $A:{\cal H}\rightrightarrows {\cal H}$ a maximally monotone operator and
$C:{\cal H}\rightarrow {\cal H}$ a monotone and $\mu$-Lipschitzian operator for $\mu>0$.
Let $m$ be a strictly positive integer and, for any $i\in\{1,\!...,m\}$, let ${\cal G}_i$  be a real Hilbert space, $r_i\in {\cal G}_i$,
let $B_i:{\cal G }_i \rightrightarrows {\cal G}_i$ be a maximally monotone operator, let
$D_i:{\cal G}_i\rightrightarrows {\cal G}_i$ be monotone such that $D_i^{-1}$ is $\nu_i$-Lipschitzian for $\nu_i>0$ and
let $L_i:{\cal H}\rightarrow$ ${\cal G}_i$ be a nonzero linear continuous operator.
The problem is to solve the primal inclusion
\begin{equation}\label{sum-k-primal-C-D}
\mbox{find } \ol x \in {\cal H} \ \mbox{such that} \ z\in A\ol x+ \sum_{i=1}^{m}L_i^*\big((B_i\Box D_i)(L_i \ol x-r_i)\big)+C \ol x,
\end{equation}
together with the dual inclusion
\begin{equation}\label{sum-k-dual-C-D}
\mbox{ find } \ol v_1 \in {\cal G}_1,\!...,\ol v_m \in {\cal G}_m \ \mbox{such that } \exists x\in {\cal H}: \
\left\{
\begin{array}{ll}
z-\sum_{i=1}^{m}L_i^*\ol v_i\in Ax+Cx\\
\ol v_i\in (B_i\Box D_i)(L_ix-r_i),\ i=1,\!...,m.
\end{array}\right.
\end{equation}
\end{problem}

We say that $(\ol x, \ol v_1,\!...,\ol v_m)\in{\cal H} \times$ ${\cal G}_1 \times...\times {\cal G}_m$ is a primal-dual solution to Problem \ref{pr1}, if

\begin{equation}\label{prim-dual-C-D}z-\sum_{i=1}^{m}L_i^*\ol v_i\in A\ol x+C\ol x \mbox{ and }\ol v_i\in (B_i\Box D_i)(L_i\ol x-r_i),\ i=1,\!...,m.\end{equation}

If $\ol x\in{\cal H}$ is a solution to \eqref{sum-k-primal-C-D}, then there exists
$(\ol v_1,\!...,\ol v_m)\in$ ${\cal G}_1 \times...\times {\cal G}_m$ such that $(\ol x, \ol v_1,\!...,\ol v_m)$ is a primal-dual solution to
Problem \ref{pr1} and, if $(\ol v_1,\!...,\ol v_m)\in$ ${\cal G}_1 \times...\times {\cal G}_m$ is a solution to \eqref{sum-k-dual-C-D}, then there
exists $\ol x\in{\cal H}$ such that $(\ol x, \ol v_1,\!...,\ol v_m)$ is a primal-dual solution to Problem \ref{pr1}. Moreover, if
$(\ol x, \ol v_1,\!...,\ol v_m)\in{\cal H} \times$ ${\cal G}_1 \times...\times {\cal G}_m$ is a primal-dual solution to Problem \ref{pr1},
then $\ol x$ is a solution to \eqref{sum-k-primal-C-D} and $(\ol v_1,\!...,\ol v_m)\in$ ${\cal G}_1 \times...\times {\cal G}_m$ is a solution
to \eqref{sum-k-dual-C-D}.\vspace{0.2cm}

Problem \ref{pr1} covers a large class of monotone inclusion problems and we refer the reader to consult \cite{combettes-pesquet}
for several interesting particular instances of it. The main result of this section follows.

\begin{theorem}\label{inertial-primal-dual} In Problem \ref{pr1} suppose that

\begin{equation}\label{ran}z\in\ran\left(A+ \sum_{i=1}^{m}L_i^*\big((B_i\Box D_i)(L_i \cdot-r_i)\big)+C\right).\end{equation}

Chose $x_0,x_1\in\cal{H}$ and $v_{i,0},v_{i,1}\in$ ${\cal G}_i$, $i=1,...,m,$ and set

$$(\forall n\geq 1)\hspace{0.2cm}\left\{
\begin{array}{rl}
p_{1,n} = & J_{\lambda_n A}[x_n-\lambda_n(Cx_n+\sum_{i=1}^mL_i^*v_{i,n}-z)+\alpha_{1,n}(x_n-x_{n-1})]\\
p_{2,i,n}= & J_{\lambda_n B_i^{-1}}[v_{i,n}+\lambda_n(L_ix_n-D_i^{-1}v_{i,n}-r_i)+\alpha_{1,n}(v_{i,n}-v_{i,n-1})],\\ 
           & i=1,...,m\\
v_{i,n+1}= & \lambda_nL_i(p_{1,n}-x_n)+\lambda_n(D_i^{-1}v_{i,n}-D_i^{-1}p_{2,i,n})+p_{2,i,n}\\
           & +\alpha_{2,n}(v_{i,n}-v_{i,n-1}), i=1,...,m\\
x_{n+1} = & \lambda_n\sum_{i=1}^mL_i^*(v_{i,n}-p_{2,i,n})+\lambda_n(Cx_n-Cp_{1,n})+p_{1,n}\\
         & +\alpha_{2,n}(x_n-x_{n-1}).
\end{array}\right.$$
Consider $\lambda, \sigma>0$ and $\alpha_1, \alpha_2\geq 0$ such that 
\begin{equation*}
12\alpha_2^2+9(\alpha_1+\alpha_2)+4\sigma<1 \ \mbox{and} \ \lambda\leq\lambda_n\leq \frac{1}{\beta}\sqrt{\frac{1-12\alpha_2^2-9(\alpha_1+\alpha_2)-4\sigma}{12\alpha_2^2+8(\alpha_1+\alpha_2)+4\sigma +2}} \ \forall n \geq 1,
\end{equation*}
where $$\beta=\max\{\mu,\nu_1,...,\nu_m\}+\sqrt{\sum_{i=1}^m\|L_i\|^2},$$
and for $i=1,2$ the nondecreasing sequences $(\alpha_{i,n})_{n \geq 1}$ with $\alpha_{i,1}=0$ and $0\leq\alpha_{i,n}\leq\alpha_i$ for all $n \geq 1$. Then the following statements are true: 
\begin{itemize}\item[(a)] $\sum_{n\in\N}\|x_{n+1}-x_n\|^2<+\infty$,
$\sum_{n \geq 1}\|x_n-p_{1,n}\|^2< +\infty$ and, for $i=1,...,m$, \newline $\sum_{n\in\N}\|v_{i,n+1}-v_{i,n}\|^2$ $< + \infty$ and $\sum_{n \geq 1}\|v_{i,n}-p_{2,i,n}\|^2<+\infty$;
\item[(b)] There exists $(\ol x, \ol v_1,\!...,\ol v_m)\in{\cal H} \times$ ${\cal G}_1 \times...\times {\cal G}_m$
a primal-dual solution to Problem \ref{pr1} such that the following hold: 
\begin{itemize}\item[(i)] $x_n\rightharpoonup\ol x$, $p_{1,n}\rightharpoonup\ol x$ and, for $i=1,...,m$, $v_{i,n}\rightharpoonup\ol v_i$ and
$p_{2,i,n}\rightharpoonup\ol v_i$ as $n \rightarrow +\infty$; 
\item[(ii)] If $A+C$ is uniformly monotone at $\ol x$, then $x_n\rightarrow\ol x$ and
$p_{1,n}\rightarrow\ol x$ as $n \rightarrow +\infty$.
\item[(iii)] If  $B_i^{-1}+D_i^{-1}$ is uniformly monotone at $\ol v_i$ for some $i\in\{1,...,m\}$, then $v_{i,n}\rightarrow\ol v_i$ and $p_{2,i,n}\rightarrow\ol v_i$ as $n \rightarrow +\infty$.
\end{itemize}
\end{itemize}
\end{theorem}

\begin{proof} We will apply Theorem \ref{inertial-tseng} in an appropriate product space and will make use to this end of  a construction similar to the one considered in \cite{combettes-pesquet}. We endow the product space 
$\fK={\cal H} \times$ ${\cal G}_1 \times...\times {\cal G}_m$ with the inner product and the associated norm defined for all
$(x,v_1,...,v_m), (y,w_1,...,w_m)\in \fK$ as
$$\langle (x,v_1,...,v_m),(y,w_1,...,w_m)\rangle_{\fK}=\langle x,y\rangle_{\cal{H}}+\sum_{i=1}^m\langle v_i,w_i\rangle_{{\cal{G}}_i}$$ and
$$\|(x,v_1,...,v_m)\|_{\fK}=\sqrt{\|x\|_{\cal{H}}^2+\sum_{i=1}^m\|v_i\|_{{\cal{G}}_i}^2},$$
respectively.

We introduce the operators $\fM:\fK\rightrightarrows\fK$,
$$\fM (x,v_1,...,v_m)=(-z+Ax)\times (r_1+B_1^{-1}v_1)\times....\times (r_m+B_m^{-1}v_m)$$ and
$\fQ: \fK\rightarrow\fK$,
$$\fQ(x,v_1,...,v_m)=\Big(Cx+\sum_{i=1}^mL_i^*v_i,-L_1x+D_1^{-1}v_1,...,-L_mx+D_m^{-1}v_m\Big)$$
and show that Theorem \ref{inertial-tseng} can be applied for the operators $\fM$ and $\fQ$ in the product space
$\fK$. Let us start by noticing that
$$\eqref{ran}\Leftrightarrow\zer(\fM+\fQ)\neq\emptyset$$ 
and
\begin{equation}\label{pr-dual-product-space}
(x,v_1,...,v_m)\in\zer(\fM+\fQ)\Leftrightarrow (x,v_1,...,v_m)\mbox{ is a primal-dual solution of Problem \ref{pr1}}.
\end{equation}
Further, since $A$ and $B_i$, $i=1,...,m$ are maximally monotone, $\fM$ is maximally monotone, too (see \cite[Props. 20.22, 20.23]{bauschke-book}). On the other hand, 
$\fQ$  is a monotone and $\beta$-Lipschitzian (see, for instance, the proof of \cite[Theorem 3.1]{combettes-pesquet}).

For every $(x,v_1,...,v_m)\in\fK$ and every $\lambda > 0$ we have (see \cite[Proposition 23.16]{bauschke-book})
$$J_{\lambda \fM}(x,v_1,...,v_m)=(J_{\lambda A}(x+\lambda z),J_{\lambda B_1^{-1}}(v_1-\lambda r_1),...,J_{\lambda B_m^{-1}}(v_m-\lambda r_m)).$$

Set
$$\fx_n=(x_n,v_{1,n},...,v_{m,n}) \ \forall n \in \N \ \mbox{and} \ \fp_n=(p_{1,n},p_{2,1,n},...,p_{2,m,n}) \ \forall n \geq 1.$$

In the light of the above considerations it follows that the iterative scheme in the statement of Theorem \ref{inertial-primal-dual} can be equivalently written as
$$\forall n\geq 1\hspace{0.2cm}\left\{
\begin{array}{ll}
\fp_n=J_{\lambda_n \fM}[\fx_n-\lambda_n\fQ\fx_n+\alpha_{1,n}(\fx_n-\fx_{n-1})]\\
\fx_{n+1}=\fp_n+\lambda_n(\fQ \fx_n-\fQ p_n)+\alpha_{2,n}(\fx_n-\fx_{n-1}),
\end{array}\right.$$ which is nothing else than the algorithm stated in Theorem \ref{inertial-tseng} formulated for the operators $\fM$ and $\fQ$.

(a) Is a direct consequence of Theorem \ref{inertial-tseng}(a).

(b)(i) Is a direct consequence of Theorem \ref{inertial-tseng}(b) and \eqref{pr-dual-product-space}.

(b)(ii) Let $n \geq 1$ be fixed. From the definition of the resolvent we get
$$\frac{1}{\lambda_n}(x_n-p_{1,n})-Cx_n-\sum_{i=1}^mL_i^*v_{i,n}+z+\frac{\alpha_{1,n}}{\lambda_n}(x_n-x_{n-1})\in Ap_{1,n}.$$
The update rule for $x_n$ yields
$$\frac{1}{\lambda_n}(p_{1,n}-x_{n+1})+Cx_n+\sum_{i=1}^mL_i^*(v_{i,n}-p_{2,i,n})+\frac{\alpha_{2,n}}{\lambda_n}(x_n-x_{n-1})=Cp_{1,n},$$
hence,
$$\frac{1}{\lambda_n}(x_n-x_{n+1})-\sum_{i=1}^mL_i^*p_{2,i,n}+z+\frac{\alpha_{1,n}+\alpha_{2,n}}{\lambda_n}(x_n-x_{n-1})\in (A+C)p_{1,n}.$$
Further, since $z-\sum_{i=1}^mL_i^*\ol v_i\in (A+C)\ol x$ and $A+C$ is uniformly monotone at $\ol x$, there exists
an increasing function $\phi_{A,C}:[0,+\infty)\rightarrow[0,+\infty]$ that vanishes only at $0$, such that
$$\left\langle \!p_{1,n}-\ol x, \frac{1}{\lambda_n}(x_n-x_{n+1})-\sum_{i=1}^mL_i^*p_{2,i,n}+z+ \!\frac{\alpha_{1,n}+\alpha_{2,n}}{\lambda_n}(x_n-x_{n-1})-\!\left(z-\sum_{i=1}^mL_i^*\ol v_i\right)\!\right\rangle$$ 
$$\geq \phi_{A,C}(\|p_{1,n}-\ol x\|),$$
thus
\begin{align}\label{A+C-unif-mon} & \frac{1}{\lambda_n}\langle p_{1,n}-\ol x, x_n-x_{n+1}\rangle+\left\langle p_{1,n}-\ol x,\sum_{i=1}^mL_i^*(\ol v_i-p_{2,i,n})\right\rangle \nonumber \\
& + \frac{\alpha_{1,n}+\alpha_{2,n}}{\lambda_n}\langle p_{1,n}-\ol x,x_n-x_{n-1}\rangle \geq \phi_{A,C}(\|p_{1,n}-\ol x\|).
\end{align}

In a similar way, for $i=1,...,m$, the definition of $p_{2,i,n}$ yields
$$\frac{1}{\lambda_n}(v_{i,n}-p_{2,i,n})+L_ix_n-D_i^{-1}v_{i,n}-r_i+\frac{\alpha_{1,n}}{\lambda_n}(v_{i,n}-v_{i,n-1})\in B_i^{-1}p_{2,i,n}$$
and from 
$$\frac{1}{\lambda_n}(p_{2,i,n}-v_{i,n+1})+L_ip_{1,n}-L_ix_n+D_i^{-1}v_{i,n}+\frac{\alpha_{2,n}}{\lambda_n}(v_{i,n}-v_{i,n-1})=D_i^{-1}p_{2,i,n}$$ 
we further obtain
$$\frac{1}{\lambda_n}(v_{i,n}-v_{i,n+1})+L_ip_{1,n}-r_i+\frac{\alpha_{1,n}+\alpha_{2,n}}{\lambda_n}(v_{i,n}-v_{i,n-1})\in (B_i^{-1}+D_i^{-1})p_{2,i,n}.$$

Moreover, since $L_i\ol x-r_i\in(B_i^{-1}+D_i^{-1})\ol v_i$, the monotonicity of $B_i^{-1}+D_i^{-1}, i=1,...,m,$ yields the inequality
$$\left\langle  \!\frac{1}{\lambda_n}(v_{i,n}-v_{i,n+1})+L_ip_{1,n}-r_i+\frac{\alpha_{1,n}+\alpha_{2,n}}{\lambda_n}(v_{i,n}-v_{i,n-1})
-(L_i\ol x-r_i),p_{2,i,n}-\ol v_i \!\right\rangle\geq 0,$$
hence \begin{eqnarray}\label{B+D-mon} & & \frac{1}{\lambda_n}\sum_{i=1}^m\langle v_{i,n}-v_{i,n+1},p_{2,i,n}-\ol v_i\rangle+\left\langle p_{1,n}-\ol x,\sum_{i=1}^mL_i^*(p_{2,i,n}-\ol v_i)\right\rangle\nonumber\\
& & +\frac{\alpha_{1,n}+\alpha_{2,n}}{\lambda_n}\sum_{i=1}^m\langle v_{i,n}-v_{i,n-1},p_{2,i,n}-\ol v_i\rangle\geq 0.\end{eqnarray}

Summing up the inequalities \eqref{A+C-unif-mon} and \eqref{B+D-mon} we obtain for all $n \geq 1$
\begin{eqnarray}\label{ABCD1}
& & \frac{1}{\lambda_n}\langle p_{1,n}-\ol x, x_n-x_{n+1}\rangle+ \frac{\alpha_{1,n}+\alpha_{2,n}}{\lambda_n}\langle p_{1,n}-\ol x,x_n-x_{n-1}\rangle\nonumber\\
& & +\frac{1}{\lambda_n}\sum_{i=1}^m\langle v_{i,n}-v_{i,n+1},p_{2,i,n}-\ol v_i\rangle+\frac{\alpha_{1,n}+\alpha_{2,n}}{\lambda_n}\sum_{i=1}^m\langle v_{i,n}-v_{i,n-1},p_{2,i,n}-\ol v_i\rangle\\
& & \geq \phi_{A,C}(\|p_{1,n}-\ol x\|)\nonumber.
\end{eqnarray}

It then follows from (a), (b)(i) and the boundedness of the sequences $(\alpha_{i,n})_{n \geq 1}$, $i=1,2$ and $(\lambda_n)_{n \geq 1}$ that
$\lim_{n\rightarrow + \infty}\phi_{A,C}(\|p_{1,n}-\ol x\|)=0$, thus $p_{1,n}\rightarrow\ol x$ as $n \rightarrow +\infty$. From (a) we get that $x_n\rightarrow\ol x$ as $n \rightarrow +\infty$.

(b)(iii) In this case one can show that instead of \eqref{ABCD1} one has for all $n \geq 1$
\begin{eqnarray}\label{ABCD2}
& & \frac{1}{\lambda_n}\langle p_{1,n}-\ol x, x_n-x_{n+1}\rangle+
\frac{\alpha_{1,n}+\alpha_{2,n}}{\lambda_n}\langle p_{1,n}-\ol x,x_n-x_{n-1}\rangle\nonumber\\
& & +\frac{1}{\lambda_n}\sum_{j=1}^m\langle v_{j,n}-v_{j,n+1},p_{2,j,n}-\ol v_j\rangle
+\frac{\alpha_{1,n}+\alpha_{2,n}}{\lambda_n}\sum_{j=1}^m\langle v_{j,n}-v_{j,n-1},p_{2,j,n}-\ol v_j\rangle\\
& & \geq\phi_{B_i^{-1},D_i^{-1}}(\|p_{2,i,n}-\ol v_i\|)\nonumber.
\end{eqnarray}
where $\phi_{B_i^{-1},D_i^{-1}}:[0,+\infty)\rightarrow[0,+\infty]$ is an increasing function that vanishes only at $0$. The same arguments as in (b)(ii) provide the desired conclusion.
\end{proof}

\begin{remark} The case $\alpha_1=\alpha_2=0$, which enforces $\alpha_{1,n}=\alpha_{2,n}=0$ for all $n \geq 1$, shows that error-free case of
the forward-backward-forward algorithm considered in \cite[Theorem 3.1]{combettes-pesquet} is a particular case of the iterative scheme introduced in Theorem \ref{inertial-primal-dual}. We refer to
Remark \ref{tseng-classic-error-free} for a discussion on how to choose the parameters $\lambda$ and $\sigma$ in order to get exactly the bounds from \cite[Theorem 3.1]{combettes-pesquet}.
\end{remark}

\section{Convex optimization problems}\label{sec-opt-pb}

The aim of this section is to show how the inertial forward-backward-forward primal-dual algorithm can be implemented when solving a primal-dual pair of convex optimization problems.

For a function $f:{\cal H}\rightarrow\overline{\R}$, where $\overline{\R}:=\R\cup\{\pm\infty\}$ is the extended real line, 
we denote by $\dom f=\{x\in {\cal H}:f(x)<+\infty\}$ its \textit{effective domain} and say that $f$ is \textit{proper} if $\dom f\neq\emptyset$ and $f(x)\neq-\infty$ for all $x\in {\cal H}$. 
We denote by $\Gamma({\cal H})$ the family of proper, convex and lower semi-continuous extended real-valued functions defined on ${\cal H}$. 
Let $f^*:{\cal H} \rightarrow \overline \R$, $f^*(u)=\sup_{x\in {\cal H}}\{\langle u,x\rangle-f(x)\}$ for all $u\in {\cal H}$, be the \textit{conjugate function} of $f$. 
The \textit{subdifferential} of $f$ at $x\in {\cal H}$, with $f(x)\in\R$, is the set $\partial f(x):=\{v\in {\cal H}:f(y)\geq f(x)+\langle v,y-x\rangle \ \forall y\in {\cal H}\}$. 
We take by convention $\partial f(x):=\emptyset$, if $f(x)\in\{\pm\infty\}$.  Notice that if $f\in\Gamma({\cal H})$, then $\partial f$ is a maximally monotone operator (see \cite{rock}) and it holds $(\partial f)^{-1} =
\partial f^*$. For two proper functions $f,g:{\cal H}\rightarrow \overline{\R}$, we consider their \textit{infimal convolution}, which is the function $f\Box g:{\cal H}\rightarrow\B$, defined by 
$(f\Box g)(x)=\inf_{y\in {\cal H}}\{f(y)+g(x-y)\}$, for all $x\in {\cal H}$.

Let $S\subseteq {\cal H}$ be a nonempty set. The \textit{indicator function} of $S$, $\delta_S:{\cal H}\rightarrow \overline{\R}$, is the function which takes the value $0$ on $S$ and $+\infty$ otherwise. 
The subdifferential of the indicator function is the \textit{normal cone} of $S$, that is $N_S(x)=\{u\in {\cal H}:\langle u,y-x\rangle\leq 0 \ \forall y\in S\}$, if $x\in S$ and $N_S(x)=\emptyset$ for $x\notin S$.

When $f\in\Gamma({\cal H})$ and $\gamma > 0$, for every $x \in {\cal H}$ we denote by $\prox_{\gamma f}(x)$ the \textit{proximal point} of parameter $\gamma$ of $f$ at $x$, 
which is the unique optimal solution of the optimization problem
\begin{equation}\label{prox-def}\inf_{y\in {\cal H}}\left \{f(y)+\frac{1}{2\gamma}\|y-x\|^2\right\}.
\end{equation}
Notice that $J_{\gamma\partial f}=(\id_{\cal H}+\gamma\partial f)^{-1}=\prox_{\gamma f}$, thus  $\prox_{\gamma f} :{\cal H} \rightarrow {\cal H}$ is a single-valued operator fulfilling the extended
\textit{Moreau's decomposition formula}
\begin{equation}\label{prox-f-star}
\prox\nolimits_{\gamma f}+\gamma\prox\nolimits_{(1/\gamma)f^*}\circ\gamma^{-1}\id\nolimits_{\cal H}=\id\nolimits_{\cal H}.
\end{equation}
Let us also recall that a proper function $f:{\cal H} \rightarrow \overline \R$ is said to be \textit{uniformly convex}, if there exists
an increasing function $\phi :[0,+\infty)\rightarrow[0,+\infty]$ which vanishes only at $0$ and such that $f(t x+(1-t)y)+t(1-t)\phi(\|x-y\|)\leq tf(x)+(1-t)f(y)$ for all $x,y\in\dom f$
and $t\in(0,1)$. In case this inequality holds for $\phi=(\beta/2)(\cdot)^2$, where $\beta >0$, then $f$ is said to be
\textit{$\beta$-strongly convex}. Let us mention that this property implies $\beta$-strong monotonicity of $\partial f$ (see \cite[Example 22.3]{bauschke-book})
(more general, if $f$ is uniformly convex, then $\partial f$ is uniformly monotone, see \cite[Example 22.3]{bauschke-book}).

Finally, we notice that for $f=\delta_S$, where $S\subseteq {\cal H}$ is a nonempty convex and closed set, it holds
\begin{equation}\label{projection}
J_{\gamma N_S}=J_{N_S}=J_{\partial \delta_S} = (\id\nolimits_{\cal H}+N_S)^{-1}=\prox\nolimits_{\delta_S}=P_S,
\end{equation}
where  $P_S :{\cal H} \rightarrow C$ denotes the \textit{projection operator} on $S$ (see \cite[Example 23.3 and Example 23.4]{bauschke-book}).

We investigate the applicability of the algorithm introduced in Section \ref{sec3} in the context of the solving of the following primal-dual pair of convex optimization problems.

\begin{problem}\label{pr5} Let ${\cal H}$ be a real Hilbert space, $z\in {\cal H}$, $f\in\Gamma({\cal H})$ and
$h:{\cal H}\rightarrow \R$ a convex and differentiable function with a $\mu$-Lipschitzian gradient for $\mu>0$. Let $m$ be a strictly positive integer and for any $i\in\{1,\!...,m\}$ let ${\cal G}_i$  be a real Hilbert space, $r_i\in {\cal G}_i$, $g_i, l_i \in\Gamma({\cal G}_i)$ such that $l_i$ is $\nu_i^{-1}$-strongly convex for $\nu_i > 0$ and $L_i:{\cal H}\rightarrow$ ${\cal G}_i$ a nonzero linear continuous operator. Consider the convex optimization problem
\begin{equation}\label{sum-k-prim-f2}
\inf_{x\in {\cal H}}\left\{f(x)+\sum_{i=1}^{m}(g_i \Box l_i)(L_ix-r_i)+h(x)-\langle x,z\rangle\right\}
\end{equation}
and its \textit{Fenchel-type dual} problem
\begin{equation}\label{sum-k-dual-f2}
\sup_{v_i\in {\cal{G}}_i,\, i=1,\!...,m}\left\{-\big(f^*\Box h^*\big)\left(z-\sum_{i=1}^{m}L_i^*v_i\right)-\sum_{i=1}^{m}\big(g_i^*(v_i)+l_i^*(v_i) + \langle v_i,r_i\rangle\big) \right\}.
\end{equation}
\end{problem}

Considering the maximal monotone operators
$$A=\partial f, C=\nabla h, B_i=\partial g_i \ \mbox{and} \ D_i=\partial l_i,\ i=1,\!...,m,$$
according to \cite[Proposition 17.10, Theorem 18.15]{bauschke-book}, $D_i^{-1} = \nabla l_i^*$ is a monotone and $\nu_i$-Lipschitzian operator for $i=1,\!...,m$.
The monotone inclusion problem \eqref{sum-k-primal-C-D} reads
\begin{equation}\label{sum-k-primal-C-D-f}
\mbox{find } \ol x \in {\cal H} \ \mbox{such that} \ z\in \partial f(\ol x)+ \sum_{i=1}^{m}L_i^*((\partial g_i \Box \partial l_i) (L_i \ol x-r_i))+ \nabla h(\ol x),
\end{equation}
while the dual inclusion problem \eqref{sum-k-dual-C-D} reads
\begin{equation}\label{sum-k-dual-C-D-f}
\mbox{ find } \ol v_1 \in {\cal G}_1,\!...,\ol v_m \in {\cal G}_m \ \mbox{such that } \exists x\in {\cal H}: \
\left\{
\begin{array}{ll}
z-\sum_{i=1}^{m}L_i^*\ol v_i\in \partial f(x)+\nabla h(x)\\
\ol v_i\in (\partial g_i \Box \partial l_i) (L_ix-r_i),\ i=1,\!...,m.
\end{array}\right.
\end{equation}

If $(\ol x, \ol v_1,\!...,\ol v_m)\in{\cal H} \times$ ${\cal{G}}_1 \times...\times {\cal{G}}_m$ is a primal-dual solution to \eqref{sum-k-primal-C-D-f}-\eqref{sum-k-dual-C-D-f}, namely,
\begin{equation}\label{prim-dual-f2}z-\sum_{i=1}^{m}L_i^*\ol v_i\in \partial f(\ol x)+\nabla h(\ol x) \mbox{ and }\ol v_i\in (\partial g_i \Box \partial l_i)(L_i\ol x-r_i),\ i=1,\!...,m,\end{equation}
then $\ol x$ is an optimal solution of the problem \eqref{sum-k-prim-f2}, $(\ol v_1,\!...,\ol v_m)$ is an optimal solution of \eqref{sum-k-dual-f2} and the optimal objective values of the two problems coincide. 
Notice that \eqref{prim-dual-f2} is nothing else than the system of optimality conditions for the primal-dual pair of convex optimization problems \eqref{sum-k-prim-f2}-\eqref{sum-k-dual-f2}.

In case a regularity condition is fulfilled, the optimality conditions \eqref{prim-dual-f2} are also necessary. More precisely,
if the primal problem \eqref{sum-k-prim-f2} has an optimal solution $\ol x$ and a suitable regularity condition is
fulfilled, then there exists $(\ol v_1,\!...,\ol v_m)$ an optimal solution to \eqref{sum-k-dual-f2} such that $(\ol x, \ol v_1,\!...,\ol v_m)$ satisfies the optimality conditions \eqref{prim-dual-f2}.

For the readers convenience, we discuss some regularity conditions which are suitable in this context. One of the weakest qualification conditions of interiority-type reads 
(see, for instance, \cite[Proposition 4.3, Remark 4.4]{combettes-pesquet})
\begin{equation}\label{reg-cond} (r_1,\!...,r_m)\in\sqri\left(\prod_{i=1}^{m}(\dom g_i + \dom l_i)-\{(L_1x,\!...,L_mx):x\in \dom f\}\right).
\end{equation}
Here, for ${\cal H}$ a real Hilbert space and $S\subseteq {\cal H}$ a convex set,  we denote by
$$\sqri S:=\{x\in S:\cup_{\lambda>0}\lambda(S-x) \ \mbox{is a closed linear subspace of} \ {\cal H}\}$$
its \textit{strong quasi-relative interior}. Notice that we always have $\inte S\subseteq\sqri S$ (in general this inclusion may be strict). If ${\cal H}$ is finite-dimensional, then $\sqri S$ coincides with $\ri S$, the relative interior of $S$, which is the interior of $S$ with respect to its affine hull.
The condition \eqref{reg-cond} is fulfilled, if: (i) for all $i=1,\!...,m$, $\dom g_i={\cal{G}}_i$ or $\dom h_i={\cal{G}}_i$, or (ii) ${\cal H}$ and ${\cal{G}}_i$ are finite-dimensional spaces
and there exists $x\in\ri\dom f$ such that $L_ix-r_i\in\ri\dom g_i+\ri\dom l_i$, $i=1,\!...,m$ (see \cite[Proposition 4.3]{combettes-pesquet}).
For other regularity conditions we refer the reader to consult \cite{b-hab, bot-csetnek, bauschke-book, bo-van, Zal-carte}.

The following statement is a particular instance of Theorem \ref{inertial-primal-dual}.

\begin{theorem}\label{inertial-primal-dual-f} 
Suppose that the primal optimization problem \eqref{sum-k-prim-f2} has an optimal solution and the regularity condition \eqref{reg-cond} is fulfilled. Chose $x_0,x_1\in\cal{H}$ and $v_{i,0},v_{i,1}\in$ ${\cal G}_i$,
$i=1,...,m,$ and set
$$(\forall n\geq 1)\hspace{0.2cm}\left\{
\begin{array}{rl}
p_{1,n}  = & \prox_{\lambda_n f}[x_n-\lambda_n(\nabla f(x_n)+\sum_{i=1}^mL_i^*v_{i,n}-z)+\alpha_{1,n}(x_n-x_{n-1})]\\
p_{2,i,n}  = & \prox_{\lambda_n g_i^*}[v_{i,n}+\lambda_n(L_ix_n-\nabla l_i^*(v_{i,n})-r_i)+\alpha_{1,n}(v_{i,n}-v_{i,n-1})],\\ 
          & i=1,...,m\\
v_{i,n+1} = & \lambda_nL_i(p_{1,n}-x_n)+\lambda_n(\nabla l_i^*(v_{i,n})-\nabla l_i^*(p_{2,i,n}))+p_{2,i,n}\\
         &+\alpha_{2,n}(v_{i,n}-v_{i,n-1}), i=1,...,m\\
x_{n+1} = & \lambda_n\sum_{i=1}^mL_i^*(v_{i,n}-p_{2,i,n})+\lambda_n(\nabla h(x_n)-\nabla h(p_{1,n}))+p_{1,n}\\
          & +\alpha_{2,n}(x_n-x_{n-1}).
\end{array}\right.$$
Consider $\lambda, \sigma>0$ and $\alpha_1\geq 0, \alpha_2\geq 0$ such that 
\begin{equation*}
12\alpha_2^2+9(\alpha_1+\alpha_2)+4\sigma<1 \ \mbox{and} \ \lambda\leq\lambda_n\leq \frac{1}{\beta}\sqrt{\frac{1-12\alpha_2^2-9(\alpha_1+\alpha_2)-4\sigma}{12\alpha_2^2+8(\alpha_1+\alpha_2)+4\sigma +2}} \ \forall n \geq 1,
\end{equation*}
where $$\beta=\max\{\mu,\nu_1,...,\nu_m\}+\sqrt{\sum_{i=1}^m\|L_i\|^2},$$
and for $i=1,2$ the nondecreasing sequences $(\alpha_{i,n})_{n \geq 1}$ with $\alpha_{i,1}=0$ and $0\leq\alpha_{i,n}\leq\alpha_i$ for all $n \geq 1$. 
Then the following statements are true:  
\begin{itemize} \item[(a)] $\sum_{n\in\N}\|x_{n+1}-x_n\|^2<+\infty$,
$\sum_{n \geq 1}\|x_n-p_{1,n}\|^2< +\infty$ and, for $i=1,...,m$, \newline $\sum_{n\in\N}\|v_{i,n+1}-v_{i,n}\|^2$ $< + \infty$ and $\sum_{n \geq 1}\|v_{i,n}-p_{2,i,n}\|^2<+\infty$;
\item[(b)] There exists $(\ol x, \ol v_1,\!...,\ol v_m)\in{\cal H} \times$ ${\cal G}_1 \times...\times {\cal G}_m$
satisfying the optimality conditions \eqref{prim-dual-f2}, hence $\ol x$ is an optimal solution of the problem \eqref{sum-k-prim-f2},
$(\ol v_1,\!...,\ol v_m)$ is an optimal solution of \eqref{sum-k-dual-f2} and the optimal objective values of the two problems coincide, such that the following hold: 
\begin{itemize}\item[(i)] $x_n\rightharpoonup\ol x$, $p_{1,n}\rightharpoonup\ol x$ and, for $i=1,...,m$, $v_{i,n}\rightharpoonup\ol v_i$ and
$p_{2,i,n}\rightharpoonup\ol v_i$ as $n \rightarrow +\infty$;  
\item[(ii)] If $f+h$ is uniformly convex, then $x_n\rightarrow\ol x$ and
$p_{1,n}\rightarrow\ol x$ as $n \rightarrow +\infty$;
\item[(iii)] If $g_i^*+l_i^*$ is uniformly convex for some $i\in\{1,...,m\}$, then $v_{i,n}\rightarrow\ol v_i$ and $p_{2,i,n}\rightarrow\ol v_i$ as $n \rightarrow +\infty$.
\end{itemize}
\end{itemize}
\end{theorem}

\begin{remark} Suppose that the primal optimization problem \eqref{sum-k-prim-f2} is feasible, which means that its optimal objective
value is not identical $+\infty$. The existence of optimal solutions of \eqref{sum-k-prim-f2} is guaranteed if
for instance, $f+h+\langle\cdot,-z\rangle$ is coercive (that is $\lim_{\|x\|\rightarrow\infty}(f+h+\langle\cdot,-z\rangle)(x)=+\infty$)
and for all $i=1.,,,.,m$, $g_i$ is bounded from below. Indeed, under these circumstances, the objective function of
\eqref{sum-k-prim-f2} is coercive (one can use \cite[Corollary 11.16 and Proposition 12.14]{bauschke-book} to show that $g_i\Box l_i$ is bounded from below and $g_i\Box l_i\in\Gamma({\cal{G}}_i)$ for $i=1,...,m$) 
and the statement follows via \cite[Corollary 11.15]{bauschke-book}. On the other hand, when $f+h$ is strongly convex, then the objective function of
\eqref{sum-k-prim-f2} is strongly convex, too, thus \eqref{sum-k-prim-f2} has a unique optimal solution (see \cite[Corollary 11.16]{bauschke-book}).
\end{remark}

\begin{remark} Let us mention that for $i\in\{1,...,m\}$ the function $g_i^*+l_i^*$ is uniformly convex, if $g_i^* + l_i^*$ is $\delta_i$-strongly convex for $\delta_i > 0$.
This is the case, for example, when $g_i^*$ (or $l_i^*$) is $\delta_i$-strongly convex or when $g_i^*$ is $\alpha_i$-strongly convex and
$l_i^*$ is $\beta_i$-strongly convex, where $\alpha_i,\beta_i>0$ are such that $\alpha_i+\beta_i\geq \delta_i$. Let us also notice that, according to
\cite[Theorem 18.15]{bauschke-book}, $g_i^*$ is $\alpha_i$-strongly convex if and only if $g_i$ is Fr\'{e}chet-differentiable  and
$\nabla g_i$ is $\alpha^{-1}_i$-Lipschitzian.
\end{remark}

\end{document}